\documentclass{amsart}

\usepackage[latin1]{inputenc}
\usepackage{amssymb}
\usepackage{amsmath}
\usepackage{amsthm,epsfig}
\usepackage{hyperref}
\usepackage{color}
\usepackage{tikz-cd}
\usepackage{multirow}
\usepackage{graphicx}
\usepackage{accents}
\usepackage{pinlabel}
\usepackage{subcaption}
\usepackage[shortlabels]{enumitem}

\def\col{\colon\thinspace}
\newcommand{\C}{{\mathbb{C}}}
\newcommand{\D}{\mbox{\rm D}}

\newcommand{\R}{{\mathbb{R}}}

\newcommand{\Z}{{\mathbb{Z}}}

\newcommand{\E}{{\mathbb{E}}}
\newcommand{\HH}{{\mathbb{H}}}

\newcommand{\sltil}{\widetilde{\mathrm{SL}}_2(\mathbb{R})}

\newcommand{\rme}{\mathrm{e}}

\newcommand{\rmi}{\mathrm{i}}

\newcommand{\vol}{\text{\rm vol}\,}
\newcommand{\tr}{\text{\rm tr}\,}

\newcommand{\abs}[1]{|#1|}
\newcommand{\set}[1]{\{#1\}}

\newcommand{\fol}{\mathcal{F}}

\newcommand{\mc}[1]{\mathcal{#1}}
\newcommand{\delsn}{\partial_{s \vert_{s=0}}}

\newcommand{\restr}[2]{\left.#1\right\vert_{#2}}

\newcommand{\drm}{\mathrm{d}}

\newcommand{\dprime}{\prime \prime}

\DeclareMathOperator{\id}{\mathrm{id}}

\DeclareMathOperator*{\arccot}{arccot}
\DeclareMathOperator*{\arcoth}{arcoth}

\DeclareMathOperator{\isom}{Isom}

\DeclareMathOperator{\nil}{Nil}

\DeclareMathOperator{\rank}{rank}

\DeclareMathOperator{\Rc}{Ric}

\DeclareMathOperator{\orb}{Orb}

\definecolor{grn}{RGB}{71,181,97}
\definecolor{ora}{RGB}{252,160,0}
\definecolor{blue}{RGB}{126,157,229}

\newtheorem{thm}{Theorem}[section]

\newtheorem{lemma}[thm]{Lemma}

\newtheorem{cor}[thm]{Corollary}
\newtheorem{prop}[thm]{Proposition}
\theoremstyle{definition}
\newtheorem{defi}[thm]{Definition}

\newtheorem{rmk}[thm]{Remark}

\newtheorem{ex}[thm]{Example}
\newtheorem*{exa}{Examples}

\title[Geodesic vector fields, induced contact structures and tightness]{Geodesic vector fields, induced contact structures and tightness in dimension three}
\author{Tilman Becker}
\address{Mathematisches Institut, Universit\"at zu K\"oln, Weyertal 86-90, 50931 K\"oln, Germany}
\email{tibecker@math.uni-koeln.de}

\begin{document}

\begin{abstract}
In this paper, we provide new and simpler proofs of two theorems of Gluck and Harrison on contact structures induced by great circle or line fibrations. Furthermore, we prove that a geodesic vector field whose Jacobi tensor is parallel along flow lines (e.g.\ if the underlying manifold is locally symmetric) induces a contact structure if the `mixed' sectional curvatures are nonnegative, and if a certain nondegeneracy condition holds. Additionally, we prove that in dimension three, contact structures admitting a Reeb flow which is either periodic, isometric, or free and proper, must be universally tight. In particular, we generalise an earlier result of Geiges and the author, by showing that every contact form on $\R^3$ whose Reeb vector field spans a line fibration is necessarily tight. Furthermore, we provide a characterisation of isometric Reeb vector fields. As an application, we recover a result of Kegel and Lange on Seifert fibrations spanned by Reeb vector fields, and we classify closed contact $3$-manifolds with isometric Reeb flows (also known as $R$-contact manifolds) up to diffeomorphism.
\end{abstract}
\maketitle
\section{Introduction}
One of the notable features of Reeb flows associated with contact forms is their \textit{geodesibility}. That is, given a contact form $\alpha$ with Reeb vector field $R$, one can find a Riemannian metric for which every flow line of $R$ is a geodesic. In fact, it is not hard to describe such a metric: start with an arbitrary metric on the contact hyperplane field $\xi = \ker \alpha$. Then, set $|R| \equiv 1$ and define $R$ to be orthogonal to $\xi$. A simple computation then shows that the flow lines of $R$ are geodesics with respect to this metric (see Lemma \ref{lem:geodesiccharact} in Section \ref{section:pre}). 
Perhaps the most prominent example is the geodesic flow on the unit tangent bundle of a manifold (Example \ref{ex:geodesicvf} (i)). Its flow lines are geodesics with respect to the Sasaki metric, and the orthogonal distribution defines the tautological contact structure. 

Conversely, if a geodesible vector field $X$ is given, say on a $3$-manifold, one can ask whether or not the distribution orthogonal to $X$ defines a contact structure, i.e.\ a maximally non-integrable plane field. In this case, we speak of \textit{contact structures induced by geodesic fields} or simply \textit{induced contact structures}. Moreover, one can ask for properties of contact structure arising this way, for example tightness in dimension 3 (see Sections \ref{section:pre} and \ref{section:tightness} for definitions). In other words, we are looking for `local-global'-type statements, where a local assumption --- typically, on the curvature of the underlying manifold --- forces some global topological property of the induced contact structure.

These types of questions have been studied by several authors. It was shown by Gluck that the distribution orthogonal to a fibration of $S^3$ by oriented great circles always defines a contact structure diffeomorphic to the standard one \cite{gluck:2018}. By work of Harrison, Geiges and the author, the same is true for fibrations of $\E^3$ by oriented lines satisfying a certain `nondegeneracy' condition \cite{harrison:2019,harrison:2021,beckergeiges:2021}. A more general result was obtained by Aazami and, independently, Harris--Paternain, who proved that a geodesic vector field $X$ on a closed $3$-manifold defines a contact structure if $\Rc X > 0$ everywhere \cite{aazami:2015, harrispaternain:2016}. It seems to be unknown, however, whether or not the contact structure must be tight in that case. On the other hand, if the sectional curvature of the underlying manifold $M$ is positive and $1/4$-pinched, and if the contact structure is `strongly' compatible with the Riemannian metric (see \cite{etnyrekomendarczykmassot:2012} for definitions), then the universal cover of $M$ is a sphere, and the contact structure lifts to the standard tight one by results of Etnyre--Komendarczyk--Massot and Ge-Huang 
\cite{etnyrekomendarczykmassot:2012,gehuang:2016}. The contact structures we are interested in are generally not strongly but only weakly compatible with the Riemannian metric. On the contrary, there are weakly compatible contact structures which are not induced by a geodesic field. In other words, our notion of induced contact structures seems to sit somewhere in between weakly and strongly compatible contact structures. While the `$1/4$-pinched contact sphere theorem' is known to be false for weakly compatible metrics in general \cite{peraltasalas:2023}, there are no counterexamples yet in the case of induced contact structures. 

In the present paper, we provide new and simpler proofs of the theorems of Gluck and Harrison (Theorem \ref{thm:contgeoconst} and Corollary \ref{cor:spaceformsinducedcs} in this paper) using a notion of \textit{adapted} Jacobi fields. Using a similar line of argument, we prove a more general statement on contact structures induced by geodesic vector fields whose Jacobi tensor is parallel along flow lines (Theorem \ref{thm:locallysymmcontact}). This applies, for example, to every geodesic vector field defined on a locally symmetric space.
In Section \ref{section:tightness}, we provide several tightness criteria for contact structures in dimension three in terms of their Reeb flows. More precisely, we prove that contact structures admitting a Reeb flow which is either periodic, isometric, or free and proper, must be universally tight (Proposition \ref{prop:periodictight} and Theorems \ref{thm:freereebflow}, \ref{prop:killingtight}). This allows us to prove tightness of contact structures induced by certain geodesic fields. In particular, we generalise an earlier result of Geiges and the author (Corollary \ref{cor:linetight}). Furthermore, we give a necessary and sufficient condition for a Killing vector field to be realisable as the Reeb vector field of a contact form (Proposition \ref{prop:killingreeb}). This yields the following two applications: a proof of an earlier result of Kegel and Lange characterising Seifert fibrations which are realisable by Reeb flows (Corollary \ref{cor:kegellange} in this paper), as well as a complete classification of closed $3$-dimensional $R$-contact manifolds, i.e.\ contact manifolds admitting an isometric isometric Reeb flow (Corollary \ref{thm:classification}).

\subsection*{Acknowledgements} This article is based on and extends parts of my Ph.D. thesis, completed at the Universit\"at zu K\"oln.  I am deeply grateful to my advisor, Hansj\"org Geiges, for his continuous support throughout my Ph.D. journey. I would also like to thank Murat Sa\u{g}lam for many helpful discussions, in particular regarding the proof of Theorem \ref{prop:killingtight}.

This work is part of a project of the SFB/TRR 191 'Symplectic Structures in Geometry, Algebra and Dynamics', funded by the DFG (Projektnummer 281071066 -- TRR 191).

\section{Preliminaries}\label{section:pre}
Let $(M,g)$ be a Riemannian manifold and $X$ a unit vector field on $M$. Then $X$ is called \textbf{geodesic} if all of its flow lines are geodesics, or equivalently, if $\nabla_X X = 0$, where $\nabla$ denotes the Levi-Civita connection associated with $g$. More generally, a vector field $X$ is called \textbf{geodesible} if there exists a Riemannian metric for which $X$ is geodesic.

Let us now discuss some first examples.
\begin{exa}\label{ex:geodesicvf}~
\begin{enumerate}[(i)]
    \item (Geodesics flows). Let $(M,g)$ be a Riemannian manifold and $(STM,g_S)$ its unit tangent bundle equipped with the Sasaki metric $g_S$, which can be described as follows. Given two unit vector fields $X$, $Y$ along curves $\gamma_X, \gamma_Y$, set
    \[
g_S(\dot{X},\dot{Y}) := g(\dot{\gamma}_X,\dot{\gamma}_Y) + g(D_t^{\gamma_X} X, D_t^{\gamma_Y} Y),
\]
where $D_t^{\gamma_X}$ and $D_t^{\gamma_Y}$ are the covariant derivatives along $\gamma_X$ and $\gamma_Y$, respectively. Now, consider the geodesic flow
\[
G \col U \times (-\varepsilon,\varepsilon) \longrightarrow STM, \quad (v,t) \longmapsto \dot{\gamma}_v(t),
\]
where $U \subset STM$ is an open, relatively compact set, and $\gamma_v$ is the unique geodesic satisfying the initial conditions $\gamma_v(0) = \pi(v)$ and $\dot{\gamma_v}(0) = v$ (here, $\pi \col STM \to M$ is the natural projection). Now let $Y$ be a vector field along a curve $\gamma$, and denote by $L(Y)$ and $L(\gamma)$ the lengths of $Y$ and $\gamma$ as curves in $STM$ and $M$, respectively. Then
\[
L(Y)  
= \int_{-\delta}^{\delta} \sqrt{|\dot{\gamma}(t)|^2_g + |D_t^{\gamma} Y|^2_g} \, \drm t \geq \int_{-\delta}^{\delta} |\dot{\gamma}(t)|_g \, \drm t = L(\gamma).
\]
with equality if and only if $D_t^{\gamma} Y \equiv 0$. If $Y$ is a flow line of $G$, then $D_t^{\gamma} Y = D_t^{\gamma} \dot{\gamma} = 0$, hence $Y$ is locally length-minimising and therefore a geodesic. Thus, $G$ is a geodesic vector field on $(STM,g_S)$.
\item (Killing flows). Let $X$ be a unit length Killing vector field on $(M,g)$, that is, the flow of $X$ defines a 1-parameter family of isometries of $M$. In particular, it preserves the orthogonal hyperplane field $X^\perp$. Hence, by Lemma \ref{lem:geodesiccharact} below, $X$ is geodesic. If $X$ is not of unit length, then after replacing $g$ by $(1/|X|^2) g$, one can compute easily that $X$ is still Killing for this new metric (see \cite[Lemma 3.1]{wadsley:1975}), hence geodesic. It follows that every Killing vector field is geodesible for a Riemannian metric conformally equivalent to the given one.
\item (Periodic flows). Consider a vector field $X$ whose flow defines an $S^1$-action on some manifold $M$. Start with an arbitrary metric $g$, then replace $g$ by 
\[
\tilde{g} = \int_{S^1} (L_{\rho})^*g \, \drm \mu,
\]
where $L_{\rho}$ denotes left multiplication by $\rho \in S^1$, and $\drm \mu$ is a Haar measure on $S^1$. This new metric is being preserved by the flow of $X$, hence $X$ is Killing and therefore geodesible by Example (ii) above.
\end{enumerate}
\end{exa}
The following characterisations were first observed by Wadsley \cite{wadsley:1975} and Sullivan \cite{sullivan:1978}.
\begin{lemma}\label{lem:geodesiccharact}
Let $X$ be a unit vector field on a Riemannian manifold $(M,g)$ with dual 1-form $\alpha := i_X g$. Then, the following are equivalent.
\begin{itemize}
    \item[(i)] $X$ is geodesic;
    \item[(ii)] the flow of $X$ preserves the orthogonal distribution $X^\perp := \ker \alpha$;
    \item[(iii)] $\alpha(X) = 1$ and $i_X \drm \alpha = 0$.
\end{itemize}
\end{lemma}
\begin{proof}
$(i) \Longleftrightarrow (iii)$: A simple computation shows that $i_X \drm \alpha = g(\nabla_X X,\cdot)$, see \cite[Lemma 3.2]{geiges:2022}. Hence, by the nondegeneracy of $g$, it follows that $i_X \drm \alpha = 0 \Leftrightarrow \nabla_X X = 0$.

$(ii) \Longleftrightarrow (iii)$: If $i_X \drm \alpha = 0$, then $L_X \alpha = 0$; hence, the flow of $X$ preserves $\alpha$ and, in particular, $X^\perp = \ker \alpha$. Conversely, if the flow of $X$ preserves $X^\perp$, then $i_X \drm \alpha = L_X \alpha = f \alpha$ for some function $f \col M \to \R$. Plugging $X$ into both sides of the equation yields $f \equiv 0$ and hence $i_X \drm \alpha = 0$. 
\end{proof}
This immediately implies the following.
\begin{cor}\label{cor:geodesiblecharact}
Let $X$ be a nowhere vanishing vector field on some manifold. Then, the following are equivalent.
\begin{itemize}
    \item[(i)] $X$ is geodesible;
    \item[(ii)] the flow of $X$ preserves some transverse distribution;
    \item[(iii)] there is a $1$-form $\alpha$ such that $\alpha(X) = 1$ and $i_X \drm \alpha = 0$. \qed
\end{itemize} 
\end{cor}
Any 1-form $\alpha$ as in Corollary \ref{cor:geodesiblecharact} is called \textbf{characteristic $1$-form} (sometimes also \textit{connection $1$-form}). 

\begin{ex}\label{ex:reeb}
Let $M$ be a $(2n+1)$-dimensional manifold and $\alpha$ a \textbf{contact form} on $M$, i.e.\ $\alpha \wedge (\drm \alpha)^n$ is nowhere vanishing. In other words, the distribution $\xi:= \ker \alpha$ is maximally non-integrable, and $\xi$ is called a \textbf{contact structure}. Associated with $\alpha$ is the \textbf{Reeb vector field} $R$, which is uniquely defined by the two equations $i_R \drm \alpha = 0$ and $\alpha(R) = 1$. In particular, $\alpha$ is a characteristic $1$-form for $R$. It follows that every Reeb vector field is geodesible.
\end{ex}

Now assume that $X$ is a geodesic vector field on a $3$-dimensional Riemannian manifold $(M, g = \langle \cdot, \cdot \rangle)$. Consider the orthogonal plane field $X^\perp$ which is being preserved by the flow of $X$ by Lemma \ref{lem:geodesiccharact}. The question we are asking now is whether or not $X^\perp$ defines a contact structure, i.e.\ whether $\alpha := i_X g$ satisfies the contact condition $\alpha \wedge \drm \alpha \neq 0$. By Lemma \ref{lem:geodesiccharact} (iii), this is equivalent to $\restr{\drm \alpha}{X^\perp}$ being nowhere vanishing. This is a local issue, so we may consider a point $p \in M$ and tangent vectors $v, w \in X^\perp_p$. Extend $v$ and $w$ arbitrarily to local vector fields $V$ and $W$. Then
\[
\begin{aligned}
\drm \alpha(V,W) &= V (\alpha(W)) - W (\alpha(V)) - \alpha([V,W]) \\
&= V\langle X,W \rangle - W \langle X,V \rangle - \langle X,[V,W] \rangle \\
&= \langle \nabla_V X, W \rangle - \langle \nabla_W X,V \rangle + \langle X,\underbrace{\nabla_V W - \nabla_W V - [V,W]}_{=0}\rangle\\[-4pt]
&= \langle \nabla_V X, W \rangle - \langle \nabla_W X,V \rangle.
\end{aligned}
\]
Evaluating this at $p$ yields
\begin{equation}\label{eq:contactcondthreefirst}
    \drm \alpha_p(v,w) = \langle (\nabla_v X)_p, w \rangle - \langle (\nabla_w X)_p,v \rangle.
\end{equation}
Now, define the linear bundle morphism
\[
\beta := \nabla X \col X^\perp \longrightarrow X^\perp, \quad v \longmapsto \nabla_v X.
\]
(The notation $\beta$ is adopted from \cite{harrispaternain:2016}).
Note that since $\langle X,X \rangle \equiv 1$ we have that $0 = X \langle X,X\rangle = 2 \langle \beta(X),X \rangle$, hence the image of $\beta$ is indeed contained in $X^\perp$. Then (\ref{eq:contactcondthreefirst}) translates into
\begin{equation}\label{eq:contactcondthree}
\drm \alpha(v,w) = \langle \beta(v),w \rangle - \langle v,\beta(w)\rangle.
\end{equation}
Clearly, the right-hand side of (\ref{eq:contactcondthree}) vanishes for all $v,w$ if and only if $\beta_p$ is self-adjoint. Hence we have proven the following.
\begin{prop}\label{prop:contactcondequiv}
    Let $X$ be a geodesic vector field on a $3$-manifold. Then $X^\perp$ is contact if and only if $\beta_p$ is not self-adjoint for every $p\in M$. \qed
\end{prop}

In case $X^\perp$ defines a contact structure, we also say that $X$ \textbf{induces} a contact structure.

\begin{rmk}\label{rmk:wholeorbitnotcontact}
\begin{enumerate}[leftmargin=*]
    \item[(i)] For a given point $p \in M$, denote by $\orb_X(p)$ the orbit of $X$ through $p$. Then Lemma \ref{lem:geodesiccharact} (iii) implies that $L_X \alpha = 0$, hence $X^{\perp}$ is contact either at every or at none of the points in $\orb_X(p)$.
    \item[(ii)] If $X$ induces a contact structure and $\alpha = i_X g$ the dual (contact) 1-form, then $\alpha(X) = 1$ and $i_X \drm \alpha = 0$ by Lemma \ref{lem:geodesiccharact}. Thus, $X$ is the Reeb vector field of $\alpha$.
\end{enumerate}

\end{rmk}

\section{Nonnegative Ricci curvature}
In this section, we will prove the following slight generalisation of the theorem of Aazami and Harris--Paternain mentioned in the introduction (see \cite[Corollary 1]{aazami:2015} and \cite[Proposition 1]{harrispaternain:2016}), in the sense that we allow the underlying manifold to be open and the Ricci curvature of $X$ to be nonnegative (instead of strictly positive). We will need this more general statement later in the proof of Theorem \ref{thm:locallysymmcontact}. I shall note, however, that the proof is essentially but a combination of the arguments in \cite{aazami:2015} and \cite{harrispaternain:2016}.

\begin{thm}\label{thm:inducedcontact}
Let $X$ be a complete geodesic vector field on a (not necessarily closed) Riemannian $3$-manifold $M$. Assume that
\[
\Rc X + \frac{\abs{\lambda-\mu}^2}{2} \geq 0
\]
everywhere, where $\lambda$ and $\mu$ are the (complex) eigenvalues of $\beta$. Then, if $X^\perp$ is not contact at $p \in M$, one of the following is true: 
\begin{enumerate}[(1)]
    \item $\Rc(X_p) < 0$, or
    \item $\Rc(X_p) = 0$ and $\beta_p = 0$.
\end{enumerate}
\end{thm}
As an immediate consequence we obtain the following corollary.
\begin{cor}\label{cor:ricpos}
Let $X$ be a complete geodesic vector field on a (not necessarily closed) Riemannian $3$-manifold $M$. Assume that $\Rc(X) \geq 0$ everywhere and $\beta_p \neq 0$ if $\Rc(X_p) = 0$. Then $X$ induces a contact structure. \qed
\end{cor}

\begin{proof}[Proof of Theorem \ref{thm:inducedcontact}]
We argue as in \cite[Proposition 1]{aazami:2015}.
Assume that $X^\perp$ is not contact at $p \in M$. Parametrise the orbit through $p$ by $t \mapsto \gamma(t)$, where $\gamma(0) = p$. by Remark \ref{rmk:wholeorbitnotcontact} (i), $X$ is not contact along all of $\gamma$. Consider now the one-parameter family $t \mapsto \beta_t := \beta_{\gamma(t)}$. Now, an easy computation shows that $\beta$ satisfies the Riccati equation
\[
\beta^\prime + \beta^2 + R_X = 0,
\]
where $\beta^\prime$ denotes the covariant derivative of $\beta$ in the direction of $X$. Taking the trace of both sides of this equation, we obtain $\tr \beta^\prime + \tr \beta^2 + \Rc X = 0$. On the other hand, a simple calculation yields $\tr \beta^\prime = X(\tr \beta) =: (\tr \beta)^\prime$, hence
\[
(\tr \beta)^\prime = - \left(\Rc (X) + \tr \beta^2 \right),
\]
see also \cite[Lemma 3]{harrispaternain:2016}. Now, since $X$ is not contact at $\gamma(t)$, it follows that $\beta_t$ is self-adjoint by Proposition \ref{prop:contactcondequiv}. That is, $\beta_t$ admits two real eigenvalues $\lambda = \lambda(t)$ and $\mu = \mu(t)$. Hence $\tr \beta^2 = \lambda^2 + \mu^2$, so that
\begin{equation}\label{eq:trbeta}
\begin{aligned}
(\tr \beta)^\prime &= - \left(\Rc (X) + \lambda^2 + \mu^2\right) \\
&= -\left(\Rc(X) + \frac{(\lambda-\mu)^2}{2} + \frac{(\tr \beta)^2}{2}\right) \\
&\leq - \frac{(\tr \beta)^2}{2}.
\end{aligned}
\end{equation}
Now, assume for the moment that there is some $t_0 \in \R$ such that $\tr \beta_{t_0} \neq 0$. After replacing $X$ by $-X$, if necessary, we may assume that $\tr \beta_{t_0} < 0$.
Let $f$ be the unique solution to the initial value problem
\begin{equation}\label{eq:odef}
    \begin{cases}
    f^\prime + \frac{1}{2}f^2 = 0, \\
    f(0) = \tr \beta_{t_0}.
    \end{cases}
\end{equation}
Then, comparing (\ref{eq:trbeta}) and (\ref{eq:odef}), we see that $\tr \beta \leq f$ everywhere, and $f$ is given by $f(t) = ((t/2)+(1/f(0)))^{-1}$. Thus $f(t) \to - \infty$ and hence $\tr \beta_t \to -\infty$ as $t \to -2/f(0)$. But this cannot happen since $X$ is assumed to be complete.
It follows that $\tr \beta_t \equiv 0$. Hence, by (\ref{eq:trbeta}), $\Rc(X) + (\lambda-\mu)^2/2 = 0$ along $\gamma$. In particular we have that either $\Rc(X_p) < 0$ or $\Rc(X_p) = 0$ and $\lambda_p = \mu_p$. In the second case, since $0 = \tr \beta_p = \lambda_p + \mu_p$, we obtain that $\lambda_p = \mu_p = 0$, thus $\beta_p = 0$ which proves the claim.
\end{proof}

\section{Adapted Jacobi fields}\label{section:jacobi}
In this section, we introduce a notion of \emph{adapted Jacobi fields}, which can be described as follows.
Let $X$ be a complete geodesic vector field on a Riemannian manifold (of any dimension). Consider a geodesic variation $\Gamma$ consisting of integral curves of $X$, i.e.\
\begin{equation}\label{eq:geodvariation}
\Gamma \col (-\varepsilon,\varepsilon) \times \R \longrightarrow M, \quad \Gamma(s,t) = \phi_t(\exp_p(sv)),
\end{equation}
where $p \in M$, $v \in T_pM$ and $\phi_t$ is the time-$t$ flow of $X$. Here, $\varepsilon$ is chosen small enough so that $\exp_p(sv)$ is defined for $|s| < \varepsilon$.
\begin{defi}
A Jacobi field $J$ is called \textbf{adapted to $X$} if it is given as the variational field of a variation through integral curves of $X$, i.e.\ $J(t) = \restr{\partial_s}{s=0} \Gamma(s,t)$, with $\Gamma$ as in (\ref{eq:geodvariation}).
\end{defi}
Adapted Jacobi fields (although not by this name) were already studied by Godoy and Salvai in \cite{godoysalvai:2015}. 
For a Jacobi field $J$, denote by $J^\prime := \nabla_X J$ the covariant derivative of $J$ in the direction of $X$.
The following was first observed in \cite{godoysalvai:2015}.
\begin{prop}\label{prop:jacobiprop}
    If $J$ is a Jacobi field adapted to a geodesic vector field, then $J^\prime = \beta(J)$.
\end{prop}
\begin{proof}
    Let $\Gamma$ be as in (\ref{eq:geodvariation}) and set $\gamma := \Gamma(0,\cdot)$. Then
    \begin{align*}
J^\prime(t) &= \D_t \, \delsn \Gamma(s,t) \\&= \D_{s \vert_{s=0}} \, \partial_t \Gamma(s,t) \\&= \D_{s \vert_{s=0}} \, X_{\Gamma(s,t)} \\&= (\nabla_{J(t)} X)_{\gamma(t)} = \beta_{\gamma(t)}(J(t)), 
    \end{align*}
    where in the second equation, we used the symmetry lemma \cite[Lemma 6.3]{lee:1997}.
\end{proof}
\begin{cor}\label{cor:jacobicommute}
    If $J$ is a Jacobi field adapted to a geodesic vector field $X$, then $J$ and $X$ commute, i.e.\ $[J,X] = 0$.
\end{cor}
\begin{proof}
    Using Proposition \ref{prop:jacobiprop}, we obtain $\nabla_X J = \beta(J) = \nabla_J X$, hence $[J,X] = \nabla_J X - \nabla_X J = 0$.
\end{proof}
\begin{cor}\label{cor:jacobicor}
    Let $J$ be a Jacobi field adapted to the geodesic vector field $X$. Then $J \equiv 0$ if and only if $J(t_0) = 0$ for some $t_0 \in \R$. 
\end{cor}
\begin{proof}
The Jacobi field $J$ solves the differential equation $J^{\prime \prime} + R(J,X)X = 0$. In particular, $J$ is determined completely by the values of $J$ and $J^\prime$ at any given point. From Proposition \ref{prop:jacobiprop} it follows that $J^\prime$ is determined by $J$, which proves the claim. \qedhere
\end{proof}
The obversations we made so far already allow for a complete characterisation of contact structures induced by Killing vector fields, as follows.
\begin{cor}
    Let $X$ be a Killing vector field of unit length on a Riemannian $3$-manifold. Then $X$ induces a contact structure if and only if $\beta \neq 0$ everywhere. In this case, the contact structure is universally tight.
\end{cor}
\begin{proof}
    The condition $\beta \neq 0$ is clearly necessary by (\ref{eq:contactcondthree}). Conversely, assume that $\beta \neq 0$ everywhere. In view of Proposition \ref{prop:contactcondequiv}, it suffices to show that if $\lambda$ is a real eigenvalue of $\beta_p$ for some $p\in M$, then $\lambda = 0$. Thus, assume that $v \in X_p^\perp$ is a (unit) eigenvector corresponding to the eigenvalue $\lambda$, and let $J$ be a Jacobi field adapted to $X$ with $J(0) = v$. Then, by Proposition \ref{prop:jacobiprop}, $J^\prime(0) = \beta(J(0)) = \lambda v$. Furthermore, $J$ is invariant under the flow of $X$ by Corollary \ref{cor:jacobicommute}. Since $X$ is Killing, this means that $|J| = \text{const.}$, hence $\langle J^\prime , J \rangle \equiv 0$. Evaluating this at $p$ yields $\lambda = 0$. This proves the first part.
    The universal tightness of the induced contact structure will be proved in a more general setting in Section \ref{section:tightness} (in fact, it will follow immediately from Proposition \ref{prop:killingtight} and Remark \ref{rmk:wholeorbitnotcontact} (ii)).
\end{proof}
\section{Space forms}
Let us now assume that $M$ is a 3-dimensional space form, i.e.\ $M$ has constant sectional curvature. Then we obtain the following.
\begin{thm}\label{thm:contgeoconst}
Let $X$ be a complete geodesic vector field on a Riemannian $3$-manifold $M$ of constant sectional curvature $c$, and $\beta_p = \nabla X \col X_p^\perp \to X_p^\perp, \, p \in M$. Then:
\begin{itemize}
\item If $c > 0$, then $\beta_p$ does not admit any real eigenvalues.
\item If $c = 0$ and $\lambda$ is a real eigenvalue of $\beta_p$, then $\lambda = 0$.
\item If $c < 0$ and $\lambda$ is a real eigenvalue of $\beta_p$, then $|\lambda| \leq \sqrt{|c|}$.
\end{itemize}
\end{thm}
\begin{proof}
Assume that $\beta_p$ has a real eigenvalue $\lambda$ for some $p\in M$, and let $v \in X_p^{\perp}$ be a corresponding eigenvector of unit length. Let $\gamma$ be the integral curve of $X$ through $p$. Consider the Jacobi field $J$ along $\gamma$ adapted to $X$ with initial condition $J(0) = v$. Then, by Proposition \ref{prop:jacobiprop}, $J^\prime(0) = \beta(J(0)) = \lambda v$. Now, since $M$ has constant sectional curvature equal to $c$, the Jacobi equation translates into $J^{\prime\prime} + c J = 0$. Hence, for $c > 0$, $J$ is given by
\[
J(t) = \left(\cos(\sqrt{c}t) + \frac{\lambda}{\sqrt{c}}\sin(\sqrt{c}t)\right) E(t),
\]
where $E$ is the parallel vector field along $\gamma$ satisfying $E(0) = v$. Now we have $J(t_0) = 0$ for 
\[
t_0 = \frac{1}{\sqrt{c}} \arccot\left(\frac{-\lambda}{\sqrt{c}}\right),
\]
hence $J \equiv 0$ by Corollary \ref{cor:jacobicor}, contradicting the fact that $J(0) = v \neq 0$.

If $c = 0$, then $J$ is given by $J(t) = (1+\lambda \, t)E(t)$ with $E$ as before. If $\lambda \neq 0$, then $J(-1/\lambda) = 0$, which contradicts Corollary \ref{cor:jacobicor} again.

Finally, assume that $c < 0$. Then $J$ is given by
\[
J(t) = \bigg(\cosh(\sqrt{|c|}t) + \frac{\lambda}{\sqrt{|c|}} \sinh(\sqrt{|c|}t) \bigg) E(t),
\]
with $E$ as before. Then, if $|\lambda| > \sqrt{|c|}$, we have that $J(t_0) = 0$ for
\[
t_0 = \frac{1}{\sqrt{|c|}} \arcoth\bigg(\frac{-\lambda}{\sqrt{|c|}}\bigg)
\]
which contradicts Corollary \ref{cor:jacobicor} again. Therefore, we must have $|\lambda| \leq \sqrt{|c|}$. This finishes the proof. 
\end{proof}
Together with Proposition \ref{prop:contactcondequiv}, we obtain the following statement, in particular reproving the results of Gluck and Harrison mentioned in the introduction.
\begin{cor}\label{cor:spaceformsinducedcs}
    Let $M$ be a Riemannian $3$-manifold of constant sectional curvature $c$ and let $X$ be a geodesic vector field on $M$. Then:
\begin{itemize}
\item If $c > 0$, then $X$ induces a contact structure.
\item If $c = 0$, then $X$ induces a contact structure if and only if $\beta$ is nowhere vanishing.
\item If $c < 0$, then $X$ induces a contact structure if and only if for all $p\in M$, there is no orthonormal basis of $X^{\perp}_p$ consisting of eigenvectors of $\beta_p$ with corresponding eigenvalues of absolute value $\leq \sqrt{|c|}$.
\end{itemize}
In any of these cases, the induced contact structure is universally tight.
\end{cor}
\begin{proof}
    If $X^\perp$ is not contact at $p$, then $\beta_p$ is self-adjoint by Proposition \ref{prop:contactcondequiv}, hence there is an orthonormal basis of $X^\perp_p$ consisting of eigenvectors of $\beta_p$ corresponding to real eigenvalues $\lambda_p$ and $\mu_p$. If $c > 0$ then this contradicts Theorem \ref{thm:contgeoconst}, hence $X^\perp$ is contact in this case. If $c = 0$ then $\lambda_p = \mu_p = 0$ by Theorem \ref{thm:contgeoconst}, hence $\beta_p = 0$. If $c < 0$, then $\abs{\lambda}, \abs{\mu} \leq \sqrt{|c|}$. 
    In the latter two cases, the conditions are clearly necessary by (\ref{eq:contactcondthree}).

    For the proof of universal tightness we again refer to Section \ref{section:tightness} (see Propositions \ref{thm:freereebflow}, \ref{prop:periodictight} and Remark \ref{rmk:wholeorbitnotcontact} (ii)).\qedhere
\end{proof}
\begin{rmk}
    Of course, the first part of Corollary \ref{cor:spaceformsinducedcs} for $c \geq 0$ also follows from Theorem \ref{thm:inducedcontact}. The proof above, however, is arguably more instructive, in particular in view of the proof of Theorem \ref{thm:locallysymmcontact} in the following section.
\end{rmk}
\section{Geodesic fields with parallel Jacobi tensor}
Given a geodesic vector field $X$, consider the \textbf{Jacobi tensor} $R_X$ defined by
\[
R_X \col X^{\perp} \longrightarrow X^{\perp}, \quad v \longmapsto R_X(v) := R(v,X)X.,
\]
where $R$ denotes the Riemann curvature tensor. It follows from the Bianchi identities that $R_{X_p}$ is a self-adjoint operator for every $p \in M$. In particular, there is an orthonormal basis $e_1,e_2$ of $X_p^\perp$ consisting of eigenvectors of $R_{X_p}$ corresponding to real eigenvalues $\Delta, \delta \in \R$, given by
\begin{equation}\label{eq:eigenvaluesrx}
\Delta = \max\limits_{v \in X^\perp_p \setminus \{0\}} K(v,X_p), \quad \delta = \min\limits_{v \in X^\perp_p \setminus \{0\}} K(v,X_p), 
\end{equation}

where $K(\cdot,\cdot)$ denotes the sectional curvature.
Extend $e_1, e_2$ to vector fields $E_1,E_2$ parallel along the integral curve of $X$ through $p$. Now consider a Jacobi field $J$ through $p$ adapted to $X$, and write $J = J_1 E_1 + J_2 E_2$. Then $J$ satisfies the equations
\begin{equation}\label{eq:jacobi}
\begin{cases}
J_1^{\prime\prime} + \Delta J_1 = 0,\\
J_2^{\prime\prime} + \delta J_2 = 0.
\end{cases}
\end{equation}
Indeed, using the fact that $R_X$ is parallel we obtain
\[
0 = (\nabla_X R_X)(E_i) = \nabla_X (R_X(E_i)) - R_X(\underbrace{\nabla_X E_i}_{= 0}) = \nabla_X (R_X(E_i)),
\]
hence the vector fields $R_X(E_i)$ are parallel, too. Since $R_X(e_1) = \Delta e_1$, it follows that $R_X(E_1) = \Delta E_1$, and similarly, $R_X(E_2) = \delta E_2$. Then $R_X(J) = \Delta J_1 E_1 + \delta J_2 E_2$, and the Jacobi equations become
\[
0 = J^{\prime\prime} + R_X(J) = \left(J_1^{\dprime} + \Delta J_1 \right) E_1 + \left(J_2^{\dprime} + \delta J_2 \right) E_2, 
\]
hence $J_1^{\dprime} + \Delta J_1 = 0$ and $J_2^{\dprime} + \delta J_2 = 0$. 

We are now ready to prove the main result of this section. This can be viewed as an extension of Corollary \ref{cor:spaceformsinducedcs} as it applies in particular to all locally symmetric spaces.
\begin{thm}\label{thm:locallysymmcontact}
    Let $X$ be a complete geodesic vector field on a Riemannian $3$-manifold $M$. Assume that $\nabla_X R_X = 0$ and for every $p \in M$, one of the following holds:
\begin{enumerate}[(i)]
    \item $\max\limits_{v \in X^\perp_p \setminus \{0\}} K(v,X_p) > 0$, or 
    \vspace{5pt}
    \item $\max\limits_{v \in X^\perp_p \setminus \{0\}} K(v,X_p) = 0$ and $\rank \beta_p = 2$,
\end{enumerate}
where $K(\cdot,\cdot)$ denotes the sectional curvature. Then X induces a contact structure.
\end{thm}

\begin{proof}[Proof of Theorem \ref{thm:locallysymmcontact}]
We argue by contradiction. Assume that $X^\perp$ is not contact at $p$. Let $\Delta$ and $\delta$ be the eigenvalues of $R_{X_p}$ as in (\ref{eq:eigenvaluesrx}). Then $\Delta \geq 0$ by our assumption on the sectional curvature. We may also assume that $\delta < 0$, for otherwise $\Rc X_p \geq 0$, in which case Corollary \ref{cor:ricpos} applies. As before, $X^\perp_p$ admits an orthonormal basis $v,w$ consisting of eigenvectors of $\beta_p$ corresponding to real eigenvalues $\lambda$ and $\mu$. Let $J$ and $\tilde{J}$ be the unique Jacobi fields through $p$ adapted to $X$ such that $J(0) = v$ and $\tilde{J}(0) = w$. Then, by Proposition \ref{prop:jacobiprop},
\[
J^\prime(0) = \beta_p(v) = \lambda v
\]
and 
\[
\tilde{J}^\prime(0) = \beta_p(w) = \mu w.
\]
Now, let us assume first that $\Delta > 0$. Writing $J = J_1 E_1 + J_2 E_2$ and $\tilde{J} = \tilde{J}_1 E_1 + \tilde{J}_2 E_2$ we obtain, using (\ref{eq:jacobi}),
\[
\begin{cases}
    J_1 = v_1 \left(\cos(\sqrt{\Delta}t) + \frac{\lambda}{\sqrt{\Delta}}\sin(\sqrt{\Delta}t)\right),\vspace{5pt}\\
    J_2 = v_2 \left(\cosh(\sqrt{|\delta|}t) + \frac{\lambda}{\sqrt{|\delta|}} \sinh(\sqrt{|\delta|}t)\right),
\end{cases}
\]
and similarly,
\[
\begin{cases}
    \tilde{J}_1 = w_1 \left(\cos(\sqrt{\Delta}t) + \frac{\mu}{\sqrt{\Delta}}\sin(\sqrt{\Delta}t)\right),\vspace{5pt}\\
    \tilde{J}_2 = w_2 \left(\cosh(\sqrt{|\delta|}t) + \frac{\mu}{\sqrt{|\delta|}} \sinh(\sqrt{|\delta|}t)\right).
\end{cases}
\]
We should stress at this point that here, $\lambda$ and $\mu$ (as well as $\Delta$ and $\delta$) are constants instead of functions of $t$ as in the proof of Theorem \ref{thm:contgeoconst}.
Now consider the function
\begin{equation}\label{eq:defphi}
\R \ni t \longmapsto A(t) := \det \begin{pmatrix} J_1(t) & \tilde{J}_1(t) \\ J_2(t) & \tilde{J}_2(t) \end{pmatrix}.
\end{equation}
Using $\sinh(t) = (\rme^t-\rme^{-t})/2$ and $\cosh(t) = (\rme^t+\rme^{-t})/2$, we compute
\[
A(t) = \frac{\rme^{\sqrt{\abs{\delta}}t}}{2}\left(c_1^+ \cos(\sqrt{\Delta}t) + c_2^+ \sin(\sqrt{\Delta}t)\right) + \frac{\rme^{-\sqrt{\abs{\delta}}t}}{2}\left(c_1^- \cos(\sqrt{\Delta}t) + c_2^- \sin(\sqrt{\delta}t)\right),
\]
where
\[
c_1^\pm = 1 \pm \frac{v_1w_2 \mu - v_2 w_1 \lambda}{\sqrt{|\delta|}}, \quad c_2^\pm = \frac{v_1w_2 \lambda - v_2 w_1 \mu}{\sqrt{\Delta}} \pm \frac{\lambda \mu}{\sqrt{\Delta |\delta|}}.
\]
Note that since $A$ does not vanish identically (since $A(0) = 1$), one of the constants $c_1^\pm, c_2^\pm$ is nonzero. Say $c_1^+ > 0$ (the other cases being similar), then $A(t_n) < 0$ for $t_n := (\pi+2\pi n)/\sqrt{\Delta}$ and $n > 0$ large. Hence, by the intermediate value theorem, there is $t_0 \in \R$ with $A(t_0) = 0$, i.e.\, $J(t_0)$ and $\tilde{J}(t_0)$ are linearly dependent. But then it follows from Corollary \ref{cor:jacobicor} that $J(t)$ and $\tilde{J}(t)$ are linearly dependent for every $t$, which is a contradiction.

Now consider the case $\Delta = 0$. In this case, it follows from (\ref{eq:jacobi}) that
\[
\begin{cases}
    J_1 = v_1 \left(1+ \lambda t\right),\vspace{5pt}\\
    J_2 = v_2 \left(\cosh(\sqrt{|\delta|}t) + \frac{\lambda}{\sqrt{|\delta|}} \sinh(\sqrt{|\delta|}t)\right),
\end{cases}
\]
as well as
\[
\begin{cases}
    \tilde{J}_1 = w_1 \left(1+ \mu t \right),\vspace{5pt}\\
    \tilde{J}_2 = w_2 \left(\cosh(\sqrt{|\delta|}t) + \frac{\mu}{\sqrt{|\delta|}} \sinh(\sqrt{|\delta|}t)\right).
\end{cases}
\]
Then
\[
A(t) = \frac{\rme^{\sqrt{\abs{\delta}}t}}{2}\left(c_1^+ + c_2^+t\right) + \frac{\rme^{-\sqrt{\abs{\delta}}t}}{2}\left(c_1^-+ c_2^- t\right),
\]
where now
\[
c_1^\pm = 1 \pm \frac{v_1w_2 \mu - v_2 w_1 \lambda}{\sqrt{|\delta|}}, \quad c_2^\pm = v_1w_2 \lambda - v_2 w_1 \mu\pm \frac{\lambda\mu}{\sqrt{\abs{\delta}}}.
\]
It follows that $c_2^+ \geq 0$ and $c_2^- \leq 0$ (otherwise, one can argue again that $A$ has to take negative values, which yields a contradiction). Hence
\[
\frac{2\lambda\mu}{\sqrt{|\delta|}} = c_2^+ - c_2^- \geq 0 \Longrightarrow \lambda \mu \geq 0.
\]
Now let $t \mapsto \gamma(t)$ be the parametrised orbit of $X$ through $p$. Then, by Remark \ref{rmk:wholeorbitnotcontact} (i), $X^\perp$ is not contact along all of $\gamma$. Hence, if $\lambda_t,\, \mu_t$ denote the eigenvalues of $\beta_t = \beta_{\gamma(t)}$, then the argument above shows that 
\begin{equation}\label{eq:lambdatmut}
    \lambda_t \mu_t \geq 0 \quad \text{for all } t \in \R.
\end{equation}
On the other hand, by Lemma \ref{lem:phi} below, $A$ is given by $A(t) = \rme^{B(t)}$, where $B$ is a primitive of $t \mapsto \tr \beta_t$. Comparing these two description yields
\begin{equation}\label{eq:phieq}
\frac{\rme^{\sqrt{\abs{\delta}}t}}{2}\left(c_1^+ + c_2^+t\right) + \frac{\rme^{-\sqrt{\abs{\delta}}t}}{2}\left(c_1^-+ c_2^- t\right) = \rme^{B(t)}.   
\end{equation}
Taking derivatives on both sides, we obtain
\[
 \frac{\rme^{\sqrt{\abs{\delta}}t}}{2}\left(\sqrt{\abs{\delta}}\left(c_1^+ + c_2^+t\right) + c_2^+\right) + \frac{\rme^{-\sqrt{\abs{\delta}}t}}{2}\left(c_2^- -\sqrt{\abs{\delta}} \left(c_1^- + c_2^- t\right)\right) = \left(\tr \beta_t\right)\rme^{B(t)},
\]
hence 
\begin{equation}\label{eq:tracebetat}
\tr \beta_t = \frac{\frac{\rme^{\sqrt{\abs{\delta}}t}}{2}\left(\sqrt{\abs{\delta}}\left(c_1^+ + c_2^+t\right) + c_2^+\right) + \frac{\rme^{-\sqrt{\abs{\delta}}t}}{2}\left(c_2^- -\sqrt{\abs{\delta}} \left(c_1^-+ c_2^- t\right)\right)}{\rme^{B(t)}}.    
\end{equation}
Recall that $c_2^+ \geq 0$ and $c_2^- \leq 0$. At this point, we have to distinguish several cases.\\[5pt]
\textbf{\underline{Case 1: $c_2^+ > 0$ and $c_2^- < 0$.}} In that case, by (\ref{eq:tracebetat}), $\tr \beta_t$ takes positive values for large positive $t$ and negative values for large negative $t$. In particular, $\tr \beta_t = 0$ for some $t$. But then $\lambda_t = -\mu_t$, hence, using (\ref{eq:lambdatmut}), 
\[
0 \leq \lambda_t\mu_t = - \lambda_t^2 \Longrightarrow \lambda_t = \mu_t = 0,
\]
so $\beta_t = 0$ which yields a contradiction.\\[5pt]
\textbf{\underline{Case 2: $c_2^+ = 0$ and $c_2^- < 0$.}} In that case we must have $c_1^+ > 0$ for otherwise $A$ takes negative values for large positive $t$, which gives a contradiction as before. But then again we have that $\tr \beta_t > 0$ for $t > 0$ large, and $\tr \beta_t < 0$ for $t<0$ large, which gives the same contradiction as in the first case.\\[5pt]
\textbf{\underline{Case 3: $c_2^+ > 0$ and $c_2^- = 0$.}} This is completely analogous to Case 2.\\[5pt]
\textbf{\underline{Case 4: $c_2^+ = 0$ and $c_2^- = 0$.}} In that case we have that
\[
0 = c_2^+ - c_2^- = \frac{2\lambda \mu}{\sqrt{|\delta|}} \Longrightarrow \lambda\mu = 0,
\]
hence $\lambda = 0$ or $\mu = 0$, so that $\rank \beta \leq 1$, contradicting the assumption. 

These are all possible cases, thus the proof is finished.\qedhere
\end{proof}
\begin{lemma}\label{lem:phi}
Assume that $X^\perp$ is not contact at $p$ and let $A$ be as in (\ref{eq:defphi}). Then 
\[
A(t) = \exp\left(\int_0^t \tr \beta(s) \drm s\right).
\]
\end{lemma}
\begin{proof}
Since $\alpha$ is invariant under the flow of $X$, the plane field $X^\perp$ is not contact along the whole orbit $\gamma$ through $p$. Hence, by Proposition \ref{prop:contactcondequiv}, for every $t \in \R$ there is an orthonormal basis $V = V_t, \, W = W_t$ of $X^\perp_{\gamma(t)}$ consisting of eigenvectors of $\beta_t$ corresponding to real eigenvalues $\lambda = \lambda_t$ and $\mu = \mu_t$, respectively. Let us assume for the moment that we can choose $V$ and $W$ to depend smoothly on $t$. Then
\[
0 = X \langle V, X \rangle = \langle \nabla_X V, X\rangle
\]
as well as 
\[
0 = X \langle V,V\rangle = 2 \langle \nabla_X V ,V \rangle.
\]
It follows that $\nabla_X V = a W$ for some function $t \mapsto a(t)$. Since
\[
0 = X \langle V, W \rangle = a + \langle V, \nabla_X W \rangle
\]
we find that $\nabla_X W = -a V$. Now write
\[
J = J_1 V + J_2 W, \quad \tilde{J} = \tilde{J}_1 V + \tilde{J}_2 W
\]
for some functions $J_i, \tilde{J_i} \col \R \to \R$.
Then
\[
\lambda J_1V + \mu J_2 W = \beta(J) = J^\prime = (J^\prime_1-aJ_2)V + (J^\prime_2 + aJ_1) W,
\]
so that
\begin{equation}
\begin{cases}
J_1^{\prime} = \lambda J_1 + a J_2.\\
J_2^{\prime} = - a J_1 + \mu J_2.
\end{cases}
\end{equation}
Similarly, we obtain
\[
\begin{cases}
\tilde{J}_1^{\prime} = \lambda \tilde{J}_1 + a \tilde{J}_2.\\
\tilde{J}_2^{\prime} = - a \tilde{J}_1 + \mu \tilde{J}_2.
\end{cases}
\]
Therefore,
\begin{equation}\label{eq:phi}
\begin{aligned}
A^\prime &= J_1^\prime \tilde{J}_2 + J_1 \tilde{J}_2^\prime - J_2^\prime \tilde{J}_1 - J_2 \tilde{J}_1^\prime\\
&= (\lambda + \mu) A(t)\\
&= (\tr \beta) A. 
\end{aligned}
\end{equation}
At this point we still have to deal with the issue of smoothness of the vector fields $V$ and $W$. It turns out that we cannot, in general, choose $V$ and $W$ to depend smoothly on $t$ for all $t \in \R$, but only on an open, dense subset $U \subset \R$ (which we are going to define in a second). But this is already good enough for our purpose, because in this case, by continuity, $A$ satisfies (\ref{eq:phi}) on \emph{all} of $\R$. Then, since $A(0) = 1$, we must have 
\[
A(t) = \exp\left(\int_0^t \tr \beta (s) \drm s\right)
\]
globally, as claimed. In order to find the subset $U$, let $U_1 := \set{t \in \R \col \lambda(t) \neq \mu(t)}$. Then, for all $t \in U_1$, there is a unique decomposition of $X^\perp_{\gamma(t)}$ into eigenspaces of $\beta$. Then $V$ and $W$ can be chosen to depend smoothly on $t$ for all $t\in U_1$ (see \cite[Chapter 9, Theorem 8]{lax:1997}). On the open subset $U_2 := \R \setminus \overline{U_1}$ we have that $\lambda(t) = \mu(t)$ and $\beta = \lambda \id$, so that every vector is an eigenvector. So on $U_2$, too, we can choose $V$ and $W$ smoothly. Then $U := U_1\sqcup U_2$ does the job. \qedhere
\end{proof}
The following two examples show that we cannot, in general, drop the assumptions $\max K(v,X) \geq 0$ or $\rank \beta = 2$ if $K(v,X) = 0$ in Theorem \ref{thm:locallysymmcontact}.
\begin{ex}\label{ex:countertolocallysymmetrictheorem}
\begin{enumerate}[(i)]
\item  Consider $M = \HH^2 \times \R$, where we use the Poincar\'e half-plane model for $\HH^2$, that is, $\HH^2 = \{(x_1,x_2) \in \R^2 \col x_2 > 0\}$ with the metric 
\[
g = \frac{\drm x_1^2 + \drm x_2^2}{x_2^2},
\]
and $M$ is equipped with the product metric. Let $x_3$ be the $\R$-coordinate and consider the geodesic vector field $X = x_2 \partial_{x_2}$. Then $X^\perp$ is spanned by $\partial_{x_1}$ and $\partial_{x_3}$, hence it is tangent to the fibration of affine planes given by $\{x_2 = \text{const.}\}$. Note that $K(\partial_{x_3},X) = 0$, so $X$ satisfies the condition $\max K(X,\cdot) = 0$ in Theorem \ref{thm:locallysymmcontact} (ii). A simple computation yields
\[
\beta(\partial_{x_1}) = - x_2^{-1}\partial_{x_1}, \quad \beta(\partial_z) = 0,
\]
so that $\rank \beta = 1$. Now since $X^\perp$ does not define a contact structure, this shows that we do need $\beta$ to have full rank in Theorem \ref{thm:locallysymmcontact} (ii).

\item\label{ex:hyperbolicgeodesicvf} Consider the half-space model of hyperbolic $3$-space, that is,
\[
\HH^3 := \{(x_1,x_2,x_3) \in \R^3 \col x_3 > 0\}
\]
equipped with the metric 
\[
g := \frac{\drm x_1^2+\drm x_2^2+\drm x_3^2}{x_3^2}.
\]
Let $X$ be the geodesic vector field on $\HH^3$ defined by $X := x_3 \, \partial_{x_3}$.
Then $X^\perp$ is spanned by $\partial_{x_1}$ and $\partial_{x_2}$ and hence tangent to the affine planes $\{x_3 = \text{const.}\}$. Furthermore, one computes $\beta(\partial_{x_1}) = - \partial_{x_1}$ and $\beta(\partial_{x_2}) = - \partial_{x_2}$, hence $\beta = - \id$ (in particular, $\rank \beta = 2$). However, $X$ does not satisfy the assumption on the sectional curvature in Theorem \ref{thm:locallysymmcontact}, since $K(\partial_{x_2},X) = K(\partial_{x_3},X) = -1 < 0$.
\end{enumerate}
\end{ex}

\section{Reeb flows and tightness}\label{section:tightness}
Recall that a contact structure $\xi$ on a $3$-manifold $M$ is called \textbf{overtwisted} if it contains an overtwisted disc; that is, if there is an embedded copy $\Delta \subset M$ of $D^2$ such that $T_p \Delta = \xi_p$ for all $p \in \partial \Delta$. If $\xi$ is not overtwisted, it is called \textbf{tight}. Moreover, $\xi$ is called \textbf{universally tight} if the pullback of $\xi$ to the universal cover of $M$ defines a tight contact structure. Note that universal tightness is in general stronger than tightness; there are examples of tight contact structures covered by overtwisted ones, also known as \textit{virtually overtwisted} contact structures.

In this section, we want to derive some tightness criteria in terms of properties of the Reeb flow associated with a contact form. A prominent example of this type is Hofer's criterion: If the Reeb flow of a contact form on a closed $3$-manifold does not admit any closed contractible orbit, then the underlying contact structure must be tight \cite{hofer:1993} (in this case also called \textit{hypertight}). The following result, which we will prove using much more elementary methods, resembles this statement in the case of $M$ being open. 

\begin{thm}\label{thm:freereebflow}
Let $(M,\xi)$ be a $3$-dimensional contact manifold without boundary. Assume that there is a contact form $\alpha$ defining $\xi$ whose Reeb flow defines a free and proper $\R$-action on $M$. Then the universal cover of $M$ is $\R^3$, and $\xi$ is universally tight.
\end{thm}
\begin{proof}
By passing to the universal cover, we may assume that $M$ is already simply connected. Since the $\R$-action defined by the flow of $R := R_{\alpha}$ is free and proper, the quotient space $\Sigma := M/\R$ is an orientable surface, and the projection $M \to \Sigma$ defines an oriented line bundle which is necessarily trivial. Hence $M \cong \Sigma \times \R$, and the flow lines of $R$ correspond to the $\R$-fibres under this identification. Note that $\int_{\Sigma} \drm \alpha \neq 0$ as $\alpha$ is a contact form. In particular, $\Sigma$ cannot be closed, for otherwise we obtain a contradiction to Stokes' theorem. Then, since $\Sigma$ must be simply connected, using the uniformisation theorem we deduce that $\Sigma$ is homeomorphic to $\R^2$, and $M \cong \R^2 \times \R = \R^3$. Choose global coordinates $(x,y,z)$ of $M = \R^2 \times \R$. Then $R$ can be written as $\lambda \, \partial_z$ for some function $\lambda \col M \to \R^+$. Now let $\Phi = \Phi_t$ denote the flow of $R$, and consider the diffeomorphism 
\[
\R^3 \longrightarrow \R^3, \quad (x,y,z) \longmapsto \Phi_z(x,y,0).
\]
The pullback of $\alpha$ by this diffeomorphism is a contact form (again called $\alpha$) whose Reeb vector field is given by $\partial_z$. Therefore, in cylindrical coordinates $(r,\theta,z)$ of $\R^3$, $\alpha$ is given by
\[
\alpha = \drm z + a \, \drm \theta + b \, \drm r
\]
for some functions $a, b \col M \to \R$. Then $i_{\partial_z} \drm \alpha = 0$ translates into $\partial_z a = \partial_z b = 0$, and we compute
\[
\alpha \wedge \drm \alpha = (\partial_r a - \partial_{\theta} b)\, \drm z \wedge \drm \theta \wedge \drm r,
\]
so that the contact condition becomes $\partial_r a - \partial_{\theta} b \neq 0$ for $r \neq 0$. Without loss of generality, we may assume that $\partial_r a - \partial_{\theta} b > 0$. 

Clearly it suffices to show that $\xi$ does not contain an overtwisted disc when restricted to a subset of the form $\{r < R \}$ with $R > 0$. Given such an $R$, consider a function $\hat{a} = \hat{a}(r) \col M \to \R^+$ such that $\hat{a}^\prime > 0$ and $\hat{a} > a$ everywhere (note that such a function exists since $a$ does not depend on $z$). Consider a bump function 
\[
\Psi = \Psi(r) \col M \longrightarrow [0,1]
\]
which is equal to $0$ on $\{r \leq R\}$ and equal to $1$ on $\{r \geq R + \varepsilon\}$ for some $\varepsilon > 0$. Now consider the 1-form 
\[
\tilde{\alpha} := \drm z + \underbrace{((1-\Psi)a + \Psi \hat{a})}_{=: \tilde{a}} \, \drm \theta + \underbrace{(1-\Psi) b}_{=:\tilde{b}} \, \drm r
\]
Note that $\partial_z \tilde{a} = \partial_z \tilde{b} = 0$, and 
\[
\partial_r \tilde{a} - \partial_{\theta} \tilde{b} = (1-\Psi)\underbrace{(\partial_r a - \partial_{\theta} b)}_{>0} + \Psi^\prime \underbrace{(\hat{a} - a)}_{> 0} + \Psi \underbrace{\hat{a}^\prime}_{>0} > 0 \quad \text{for } r > 0,
\]
hence $\tilde{\alpha}$ is a contact form on $\{r > 0\}$. It also also contact at the origin, since $\tilde{\alpha} = \alpha$ on $\{r \leq R\}$. The Reeb vector field of $\tilde{\alpha}$ is given by $R_{\tilde{\alpha}} = \partial_z$.

Now consider a function $h = h(r) \col M \to \R$ satisfying the following:
\begin{itemize}
    \item $h^\prime(r) > 0$ for all $r$;
    \item $h(r) = r^2$ for $r$ close to $0$;
    \item $h(r) = \tilde{a}(r)$ for $r \geq R + \varepsilon$.
\end{itemize}
Then, the 1-form $\beta := \drm z + h \, \drm \theta$ is a contact form with the same Reeb vector field as $\tilde{\alpha}$, namely $R_{\beta} = \partial_z$. We now want to show that $\tilde{\alpha}$ and $\beta$ are diffeomorphic contact forms, using a Moser type argument. Consider the $1$-parameter family of 1-forms
\[
\alpha_t = (1-t) \, \tilde{\alpha} + t \beta, \quad t \in [0,1].
\]
Using $R_{\tilde{\alpha}} = R_{\beta} = \partial_z$, we compute
\[
\alpha_t \wedge \drm \alpha_t(\partial_x,\partial_y,\partial_z) = \drm \alpha_t(\partial_x,\partial_y) = (1-t) \underbrace{\drm\tilde{\alpha}(\partial_x,\partial_y)}_{>0} + t \underbrace{\drm \beta(\partial_x,\partial_y)}_{>0} > 0,
\]
hence $\alpha_t$ is a contact form for every $t \in [0,1]$. We now want to find an isotopy $\Phi_t$, given by the time-$t$ flow of some vector field $X_t$, such that 
\begin{equation}\label{eq:phit}
\Phi_t^*\alpha_t = \tilde{\alpha}.
\end{equation}
In particular, this will imply that $\tilde{\alpha}$ and $\beta$ are diffeomorphic, as $\tilde{\alpha} =  \Phi_1^*\alpha_1 = \Phi_1^* \beta$. In order to find $X_t$, we differentiate equation (\ref{eq:phit}) with respect to $t$ to obtain
\[
0 = \Phi_t^*(\dot{\alpha_t} + L_{X_t} \alpha_t) = \Phi_t^*(\dot{\alpha_t} + \drm (\alpha_t(X_t)) + i_{X_t} \drm \alpha_t).
\]
Thus it suffices to find $X_t$ satisfying 
\[
\dot{\alpha_t} + \drm (\alpha_t(X_t)) + i_{X_t} \drm \alpha_t = 0.
\]
Choosing $X_t \in \ker \alpha_t$, this translates into
\begin{equation}\label{eq:moser}
i_{X_t} \drm \alpha_t = \tilde{\alpha} - \beta
\end{equation}
Now since $R_{\alpha_t} = \partial_z$ for all $t$, it follows that $\dot{\alpha_t}(R_{\alpha_t}) = (\beta-\tilde{\alpha})(\partial_z) = 0$, hence equation (\ref{eq:moser}) has a unique solution $X_t \in \ker \alpha_t$ by the nondegeneracy of $\drm \alpha_t$. Note that $X_t$ is of the form $X^1_t \, \partial_{\theta} + X^2_t \, \partial_r$ for some functions $X^1_t,X^2_t \col M \to \R$ whose support is contained in $\{r \leq R + \varepsilon\}$, since 
\[
\restr{\left(\tilde{\alpha}-\beta\right)}{\{r \geq R +\varepsilon\}} = 0. 
\]
Thus, the time-$t$ flow of $X_t$ exists for all $t$, which yields the desired isotopy $\Phi_t$. It follows that $\tilde{\alpha}$ and $\beta$ are diffeomorphic contact forms. On the other hand, by our assumptions on the function $h$, we can find a diffeomorphism $\rho \col [0,\infty) \to [0,\infty)$ (with $\rho = \id$ close to $0$) such that $(h \circ \rho)(r) = r^2$. Hence, the pullback of $\beta$ under the diffeomorphism $(r,\theta,z) \mapsto (\rho(r),\theta,z)$ of $\R^3$ is given by $\drm z + r^2 \, \drm \theta$, the standard contact form on $\R^3$, whose kernel defines a tight contact structure (see \cite[Cor. 6.5.10]{geiges:2008}). This shows that $\restr{\xi}{\{r < R\}}$ is tight and thus proves the claim. \qedhere

In particular, together with Remark \ref{rmk:wholeorbitnotcontact} (ii), we obtain the following generalisation of \cite[Theorem 1]{beckergeiges:2021}.
\begin{cor}\label{cor:linetight}
    Let $\alpha$ be a contact form on $\R^3$ whose Reeb vector field spans a fibration by lines. Then $\ker \alpha$ is tight. In particular, every contact structure induced by a line fibration of $\R^3$ is tight.\qed
\end{cor}

\end{proof}
The following statement is probably known, but --- to the best of the author's knowledge --- nowhere stated explicitly in the literature. Hence we will provide a proof.
\begin{prop}\label{prop:periodictight}
    Let $\xi$ be a contact structure on a $3$-manifold $M$ admitting a Reeb vector field with closed orbits only. Then $\xi$ is universally tight.
\end{prop}
\begin{proof}
Let $R$ be a Reeb vector field of a contact form $\alpha$ defining $\xi$ such that $R$ has closed orbits only. In particular, $M$ is Seifert fibred, and the universal cover $\tilde{M}$ of $M$ is diffeomorphic to one of $S^3$, $\R^3$ or $S^2 \times \R$. Denote by $\tilde{\alpha}$ the pullback of $\alpha$ to $M$, and by $\tilde{R} = R_{\tilde{\alpha}}$ its Reeb vector field (given by the pullback of $R$). If $\tilde{M} = S^2 \times \R$, then the Seifert fibres of $M$ are covered by the $\R$-fibres (see \cite[Lemma 3.1]{scott:1983}). In particular, $\tilde{R}$ is transverse to $S^2 \times \{0\}$. But then $\drm \alpha$ restricts to an exact area form on the closed surface $S^2 \times \{0\}$, which contradicts Stokes' theorem. Hence we only have to consider the two cases $\tilde{M} = S^3$ and $\tilde{M} = \R^3$. 
\\[3pt]
\underline{\textbf{Case $\tilde{M} = S^3$}}: The vector field $\tilde{R}$ spans a Seifert fibration of $S^3$, whose fibres are given by the orbits of an $S^1$-action of the form
\[
\C^2 \supset S^3 \ni (z_1,z_2) \longmapsto \theta(z_1,z_2):= (e^{k_1\rmi \theta}z_1,e^{k_2 \rmi \theta} z_2), \quad \theta \in S^1 = \R / 2 \pi \Z,
\]
for some $k_1, k_2 \in \Z \setminus \{0\}$ (see \cite[Proposition 5.2]{geigeslange:2018}). It follows that $\tilde{R}$ is --- up to multiplication by a nonzero function --- equal the Reeb vector field of the contact form
\[
\alpha_{k_1,k_2} := \frac{1}{k_1}(x_1 \, \drm y_1 - y_1 \, \drm x_1) + \frac{1}{k_2}(x_2 \, \drm y_2 - y_2 \, \drm x_2),
\]
where $z_1 = x_1 + \rmi y_1$ and $z_2 = x_2 + \rmi y_2$.
At this point, it suffices to argue that $\xi_{k_1,k_2} := \ker \alpha_{k_1,k_2}$ defines a tight contact structure, since any two contact structures with orbit equivalent Reeb vector fields on a closed $3$-manifold are diffeomorphic \cite[Lemma 5]{becker:2023}. 
The contact structure $\xi_{k_1,k_2}$, however, is diffeomorphic to the standard (tight) one on $S^3$ --- indeed, after applying the diffeomorphisms $h^{\pm} \col (z_1,z_2)\mapsto(\pm z_1,\mp z_2)$, if necessary, we may assume that $k_1,k_2 > 0$. Then we apply Gray stability to the $1$-parameter family of contact forms
\[
\alpha_t := \frac{1-t(1-k_1)}{k_1}\left(x_1 \, \drm y_1 - y_1 \, \drm x_1\right) + \frac{1-t(1-k_2)}{k_2}(x_2 \, \drm y_2 - y_2 \, \drm x_2), \quad t \in [0,1]
\]
to see that $\xi_{k_1,k_2}$ is diffeomorphic to $\ker \alpha_1$, the standard contact structure on $S^3$.
\\[3pt]
\underline{\textbf{Case $\tilde{M} = \R^3$:}} In this case, $\tilde{M}$ is the total space of a principal (hence trivial) $\R$-bundle whose fibres are spanned by $\tilde{R}$. In particular, the flow of $\tilde{R}$ defines a free and proper $\R$-action. Thus, $\ker \tilde{\alpha}$ is tight by Theorem \ref{thm:freereebflow}. \qedhere
\end{proof}

With the help of Proposition \ref{prop:periodictight}, we are able to prove the following.
\begin{thm}\label{prop:killingtight}
Let $(M,\xi)$ be a closed contact $3$-manifold admitting a Reeb flow which is Killing for some metric on $M$. Then $\xi$ is universally tight.
\end{thm}
\begin{rmk}
    A contact manifold admitting an isometric Reeb flow is also called \textbf{$R$-contact manifold}, cf.\ \cite{rukimbira:1993}.
\end{rmk}
Before proving Theorem \ref{prop:killingtight}, let us set up some notation. Given a closed Riemannian 3-manifold $M$, denote by $\mathfrak{K}(M)$ the space of Killing vector fields on $M$ equipped with the supremum norm given by $||X||_{\infty} := \sup_{p\in M} |X(p)|$ for $X \in \mathfrak{K}(M)$. By the classical Myers--Steenrod theorem, the isometry group $\isom(M)$ of $M$ is a compact Lie group. Denote by $\mathfrak{isom}(M)$ its Lie algebra. Then there is a natural isomorphism
\[
\mathfrak{isom}(M) \overset{\cong}{\longrightarrow} \mathfrak{K}(M).
\]
We pull back the norm on $\mathfrak{K}(M)$ via this isomorphism to obtain a norm on $\mathfrak{isom}(M)$, which we will denote by $||\cdot||_{\infty}$ again.
\begin{proof}[Proof of Theorem \ref{prop:killingtight}]
Let $R$ be an isometric Reeb vector field of some contact form $\alpha$ defining $\xi$. Without loss of generality, we may assume that $R$ is orthogonal to $\xi$. Indeed, if that were not the case, we simply replace the Riemannian metric $g$ by a new metric $\tilde{g}$, defined by $\restr{\tilde{g}}{\xi} = \restr{g}{\xi}$, $\tilde{g}(R,R) = 1$ and $\tilde{g}(R,v) = 0$ for all $v \in \xi$. Then $R$ is still Killing with respect to $\tilde{g}$, as the flow of $R$ preserves $\xi$. Furthermore, $R$ is geodesic with respect to $\tilde{g}$ (see Example \ref{ex:geodesicvf} (ii)), and equal to the Reeb vector field of the contact form $\alpha = i_R \tilde{g}$ by Remark \ref{rmk:wholeorbitnotcontact} (ii). We continue to work with this new metric and call it $g$ again. The flow of $R$ defines a $1$-parameter subgroup $G \subset \isom(M)$ of the (closed) isometry group of $M$. The closure $H:= \overline{G}$ of $G$ is abelian and hence isomorphic to a torus. Note that $H$ acts on $(M,\alpha)$ by strict contactomorphisms: Indeed, since $H$ is abelian, we have that $\phi_t \circ h = h \circ \phi_t$ for every $h \in H$, where $\phi_t$ denotes the time-$t$ flow of $R$. Differentiating this equation with respect to $t$ yields $h_* R = R$. Using the fact that $\alpha = i_R g$, a simple computation then shows that $h^* \alpha = \alpha$ (see \cite[Proposition 1]{banyagarukimbira:1995}).
Via the identification of $\mathfrak{K}(M)$ with $\mathfrak{isom}(M)$, the vector field $R$ corresponds to an element $r \in \mathfrak{h} \subset \mathfrak{isom}(M)$, where $\mathfrak{h}$ denotes the Lie algebra of $H$. Choose an element $r^\prime \in B_{1/2}(r) \cap \mathfrak{h}$ corresponding to a periodic vector field $R^\prime$, where $B_{1/2}(r)$ denotes the open ball of radius $1/2$ with respect to the norm $||\cdot||_{\infty}$ on $\mathfrak{isom}(M)$. This is possible since the set of such elements is dense in $\mathfrak{h}$. Since $||R-R^\prime||_{\infty} < 1/2$ and $|R| \equiv 1$, it follows that $R^\prime$ is nowhere vanishing and $\alpha(R^\prime) > 0$. Now consider the contact form $\alpha^\prime := (1/(\alpha(R^\prime)) \alpha$ defining $\xi$. Then 
\[
\drm \alpha^\prime = \frac{1}{\alpha(R^\prime)} \drm \alpha - \frac{1}{\alpha(R^\prime)^2} \drm (i_{R^\prime} \alpha) \wedge \alpha = \frac{1}{\alpha(R^\prime)} \drm \alpha + \frac{1}{\alpha(R^\prime)^2} (i_{R^\prime} \drm \alpha) \wedge \alpha,
\]
where we used the fact that $L_{R^\prime} \alpha = 0$. Hence $i_{R^\prime} \drm \alpha^\prime = 0$, and, of course, $\alpha^\prime(R^\prime) = 1$. Thus, $R^\prime$ is the Reeb vector field of $\alpha^\prime$, and therefore, $\xi$ is universally tight by Proposition \ref{prop:periodictight}. \qedhere
\end{proof}

Our next goal is to give a complete characterisation of isometric Reeb vector fields on closed $3$-manifolds. To do so, we need some preparation. Given a nowhere vanishing vector field $X$ on some manifold $M$ (of arbitrary dimension), consider the one-dimensional foliation $\fol_X$ spanned by $X$. A $k$-form $\eta$ is called \textbf{basic} with respect to $\fol_X$ if $i_X \eta = 0$ and $L_X \eta = 0$ (or, equivalently, $i_X \drm \eta = 0$). Denote by $\Omega^k(\fol_X)$ the space of all basic $k$-forms. Then, since the exterior derivative preserves basic forms, there is a subcomplex
\[
\ldots \xrightarrow{\drm} \Omega_B^{k-1}(\mc{F}) \xrightarrow{\drm} \Omega_B^k(\mc{F}) \xrightarrow{\drm} \Omega_B^{k+1}(\mc{F}) \xrightarrow{\drm} \ldots
\]
of the de Rham complex of $M$. Its cohomology is called \textbf{basic cohomology} of $\fol_X$ and denoted by $H_B^{\bullet}(\fol_X)$. Now, if $X$ is a geodesible vector field and $\alpha$ a characteristic 1-form for $X$, then the 2-form $\drm \alpha$ is basic and thus defines a basic cohomology class $[\drm \alpha] \in H_B^2(\fol_X)$. It is a simple exercise to show that this class does not depend on the choice of $\alpha$. Hence one can define the \textbf{Euler class} of $X$ by $e_X := [\drm \alpha] \in H^2_B(\fol_X)$ (cf.\ \cite{geiges:2022}). Then we have the following criterion for the `Reebability' of a geodesible field in terms of its Euler class due to Geiges.
\begin{prop}[{\cite[Proposition 5.5.]{geiges:2022}}]\label{prop:geiges}
    A geodesible vector field $X$ on a $(2n+1)$-dimensional manifold can be realised by the Reeb vector field of a contact form if and only if $e_X$ has an odd-symplectic representative, i.e.\ if there is a closed basic $2$-form $\omega \in \Omega^2_B(\fol_X)$ such that $[\omega] = e_X$ and $\omega^n \neq 0$. \qed
\end{prop}
Furthermore, if $M$ is a closed $3$-manifold, one can define the \textbf{volume} of $X$ as
\[
\vol_X := \int_M \alpha \wedge \drm \alpha,
\]
where $\alpha$ is any characteristic $1$-form. Again, this quantity only depends on the vector field $X$, and not on the choice of characteristic 1-form; this follows from the identity
\[
\alpha \wedge \drm \alpha - \beta \wedge \drm \beta = (\alpha - \beta)\wedge (\drm \alpha + \drm \beta) + \drm(\alpha \wedge \beta)
\]
that holds for any two $1$-forms $\alpha$ and $\beta$, see \cite[Proposition 1.2]{geiges:2022}. Note that if $X$ is the Reeb vector field of a contact form, then $\vol_X \neq 0$ as $\alpha \wedge \drm \alpha$ defines a volume form. It turns out that for Killing vector fields, the converse is also true.
\begin{prop}\label{prop:killingreeb}
    Let $X$ be a nowhere vanishing Killing vector field on a closed orientable Riemannian $3$-manifold. Then $X$ is the Reeb vector field of a contact form if and only if $\vol_X \neq 0$.
\end{prop}
\begin{proof}
    Let $\fol_X$ be the one-dimensional foliation spanned by $X$. Since $\vol_X \neq 0$, the Euler class $e_X$ of $X$ must be nontrivial. Indeed, if $e_X = 0$, then $\drm \alpha = \drm \beta$ for some basic $1$-form $\beta$. Hence $\beta \wedge \drm \alpha = 0$ (which can be seen taking the interior product with $X$). Thus
    \[
\vol_X = \int_M \alpha \wedge \drm \alpha = \int_M \alpha \wedge \drm \beta = -\int_M \beta \wedge \drm \alpha = 0.
    \]
Now choose an orientation for $M$, and consider the Riemannian volume form $\mu$ which is preserved by the flow of $X$. Then the $2$-form $\omega := i_X \mu$ is basic, so it defines a basic cohomology class $[\omega] \in H^2_B(\fol_X)$. Applying the same reasoning as for $e_X$, we see that $[\omega]$ is nontrivial. Furthermore, since $X$ is Killing we have that $H^2_B(\fol_X) \cong \R$ (see \cite[(6.5)]{tondeur:1997}). These two observations imply that $H^2_B(\fol_X) = \{c [\omega] \col c \in \R \}$, hence $e_X = c [\omega] = [c \, \omega]$ for some $c \in \R \setminus \{0\}$. But $c \, \omega$ restricts to a symplectic form on $X^\perp$, hence $X$ is Reeb by Proposition \ref{prop:geiges}.\qedhere
\end{proof}
As a consequence of Proposition \ref{prop:killingreeb}, we recover the following result by Kegel and Lange.
\begin{cor}[{\cite[Theorem 1.4]{kegellange:2021}}]\label{cor:kegellange}
    Consider a Seifert fibration $\pi \col M \to B$ with Euler number $e$. Then there is a contact form whose Reeb vector field spans the fibres of $\pi$ if and only if $e \neq 0$.
\end{cor}
\begin{proof}
    Choose a vector field $X$ defining the Seifert fibration, such that the regular fibres have minimal period $1$. Then the flow of $X$ defines an $S^1$-action on $M$, hence there is a Riemannian metric on $M$ for which $X$ is Killing (see Example \ref{ex:geodesicvf} (ii)). In particular, by Proposition \ref{prop:killingreeb}, $X$ is Reeb if and only if $\vol_X \neq 0$. The claim then follows from the identity $\vol_X = -e$ (see \cite[Corollary 6.3]{geiges:2022}).
\end{proof}

We conclude with the following classification of closed $3$-dimensional $R$-contact manifolds (see, for example, \cite{scott:1983} for definitions of the Thurston geometries).
\begin{cor}\label{thm:classification}
    Let $M$ be a closed $3$-dimensional $R$-contact manifold. Then $M$ is diffeomorphic to $\tilde{M} / \Gamma$, where $\tilde{M}$ is one of the three Thurston geometries $S^3$, $\nil$ and $\sltil$, and $\Gamma$ is a discrete subgroup of the isometry group of $\tilde{M}$ acting properly discontinuously on $\tilde{M}$.
\end{cor}
\begin{rmk}
    Note that we have the following hierachy of contact $3$-manifolds:
\[
\text{Sasakian} \subset \text{$K$-contact} \subset \text{$R$-contact}
\]
(see for example \cite{blair:2001} for the definitions of Sasakian and $K$-contact structures). In particular, Corollary \ref{thm:classification} recovers the classification of Sasakian $3$-manifolds by Geiges \cite{geiges:1997}. Moreover, it was shown in \cite{geiges:1997} that every manifold whose diffeomorphism type is as in Corollary \ref{thm:classification} admits a Sasakian structure; hence, for any closed $3$-manifold $M$, we have the following equivalent statements:
\begin{center}
    $M$ admits a Saskian structure $\Longleftrightarrow$ $M$ admits a $K$-contact structure\\
    $\Longleftrightarrow M$ admits an $R$-contact structure.

\end{center}
\end{rmk}
\begin{proof}[Proof of Corollary \ref{thm:classification}]
Let $R$ be an isometric Reeb vector field on $M$. As in the proof of Theorem \ref{prop:killingtight}, we find a nowhere vanishing isometric Reeb vector field $R^\prime$ whose flow defines an $S^1$-action on $M$. In particular, $M$ is Seifert fibred. Let $e$ be the Euler number of this Seifert fibration, and let $\tau >0$ denote the minimal period of the regular fibres. Then, by \cite[Corollary 6.3]{geiges:2022},
\[
0 \neq \vol_{R^\prime} = \tau^2 e,
\]
hence $e \neq 0$. The statement then follows from a classical result of Scott \cite[Theorem 5.3]{scott:1983} (see also \cite[Chapter 4]{becker:thesis}). \qedhere
\end{proof}


\begin{thebibliography}{72}
\bibitem{aazami:2015} \textsc{A. B. Aazami}, The Newman-Penrose Formalism for Riemannian 3-manifolds, \textit{J. Geom. Phys.} \textbf{94} (2015), 1--7.

\bibitem{banyagarukimbira:1995} \textsc{A. Banyaga and P. Rukimbira}, On characteristics of circle invariant presymplectic forms, \textit{Proc. Amer. Math. Soc.} \textbf{123} (1995), 3901--3906.

\bibitem{becker:2023} \textsc{T. Becker}, Geodesic and conformally {R}eeb vector fields on flat 3-manifolds, \textit{Differential Geom. Appl.} \textbf{89} (2023), Paper No. 102013, 18 pp.

\bibitem{becker:thesis} \textsc{T. Becker}, On geodesible vector fields and related geometric structures, Ph.D. thesis, Universit\"at zu K\"oln (2023), available at: \href{https://kups.ub.uni-koeln.de/71821/}{https://kups.ub.uni-koeln.de/71821/}

\bibitem{beckergeiges:2021} \textsc{T. Becker and H. Geiges}, The contact structure induced by a line fibration of $\R^3$ is standard, \textit{Bull. Lond. Math. Soc.} \textbf{53} (2021), 104--107.


\bibitem{blair:2001} \textsc{D. Blair}, \textit{Riemannian Geometry of Contact and Symplectic Manifolds}, Progr. Math. \textbf{203}, Birkh\"auser Boston, Inc., Boston, MA (2002).

\bibitem{epstein:1972} \textsc{D. B. A. Epstein}, Periodic flows on three-manifolds, \textit{Ann. of Math. (2)} \textbf{95} (1972), 66--82.

\bibitem{etnyrekomendarczykmassot:2012} \textsc{J. Etnyre, R. Komendarczyk and P. Massot}, Tightness in contact metric 3-manifolds, \textit{Invent. Math.} \textbf{188} (2012), 621--657.


\bibitem{gehuang:2016} \textsc{J. Ge and Y. Huang}, $1/4$-pinched contact sphere theorem, \textit{Asian J. Math.} \textbf{20} (2016), 893--901.

\bibitem{geiges:1997} \textsc{H. Geiges}, Normal contact structures on $3$-manifolds, \textit{Tohoku Math. J.} \textbf{49} (1997), 415--422.

\bibitem{geiges:2008} \textsc{H. Geiges}, \textit{An Introduction to Contact Topology}, Cambridge Stud. Adv. Math. \textbf{109}, Cambridge University Press, Cambridge (2008).

\bibitem{geiges:2022} \textsc{H. Geiges}, What does a vector field know about volume?, \textit{J. Fixed Point Theory Appl.} \textbf{24} (2022), Paper No. 23, 26 pp.

\bibitem{geigeslange:2018} \textsc{H. Geiges and C. Lange}, Seifert fibrations of lens spaces, \textit{Abh. Math. Semin. Univ. Hambg.} \textbf{88} (2018), 1--22.

\bibitem{gluck:1980} \textsc{H. Gluck}, Dynamical behavior of geodesic fields, in: \textit{Global Theory of Dynamical Systems}, Lecture Notes in Math. \textbf{819}, Springer-Verlag, Berlin (1980), 190--215.

\bibitem{gluck:2018} \textsc{H. Gluck}, Great circle fibrations and contact structures on the 3-sphere, \textit{Geom. Dedicata} \textbf{216} (2022), Paper No. 72, 15 pp.


\bibitem{godoysalvai:2015} \textsc{Y. Godoy and M. Salvai}, Global smooth geodesic foliations of the hyperbolic space, \textit{Math. Z.} \textbf{281} (2015), 43--54.

\bibitem{gotayetal:1983} \textsc{M. Gotay, R. Lashof, J. \'{S}niatycki, A. Weinstein}, Closed forms on symplectic fibre bundles, \textit{Comment. Math. Helv.} \textbf{58} (1983), 617--621.

\bibitem{harrispaternain:2016} \textsc{A. Harris and G. Paternain}, Conformal great circle flows on the 3-sphere, \textit{Proc. Amer. Math. Soc.} \textbf{144} (2016), 1725--1734.

\bibitem{harrison:2019} \textsc{M. Harrison}, Contact structures induced by skew fibrations of $\mathbb{R}^3$, \textit{Bull. Lond. Math. Soc.} \textbf{51} (2019), 887--899.

\bibitem{harrison:2021} \textsc{M. Harrison}, Fibrations of $\mathbb{R}^3$ by oriented lines, \textit{Algebr. Geom. Topol.} \textbf{21} (2021), 2899--2928.

\bibitem{hofer:1993} \textsc{H. Hofer}, Pseudoholomorphic curves in symplectizations with applications to the Weinstein conjecture in dimension three, \textit{Invent. Math.} \textbf{114} (1993), 515--563.

\bibitem{kegellange:2021} \textsc{M. Kegel and C. Lange}, A Boothby-Wang theorem for Besse contact manifolds, \textit{Arnold Math. J.} \textbf{7} (2021), 225--241.

\bibitem{lax:1997} \textsc{P. D. Lax}, \textit{Linear Algebra}, Pure Appl. Math. (N. Y.), John Wiley {\&} Sons, Inc., New York (1997).


\bibitem{lee:1997} \textsc{J. M. Lee}, \textit{Riemannian Manifolds - An Introduction to Curvature}, Grad. Texts Math. \textbf{176}, Springer, Berlin (1997).




\bibitem{peraltasalas:2023} \textsc{D. Peralta-Salas and R. Slobodeanu}, Contact structures and Beltrami fields on the torus and the sphere, \textit{Indiana Univ. Math. J.} \textbf{72} (2023), 699--730.

\bibitem{rukimbira:1993} \textsc{P. Rukimbira}, Some remarks on $R$-contact flows, \textit{Ann. Global Anal. Geom.} \textbf{11} (1993), 165--171.

\bibitem{scott:1983} \textsc{P. Scott}, The geometries of 3-manifolds, \textit{Bull. Lond. Math. Soc.} \textbf{15} (1983), 401--487.

\bibitem{sullivan:1978} \textsc{D. Sullivan}, A foliation of geodesics is characterized by having no ``tangent homologies'', \textit{J. Pure Appl. Algebra} \textbf{13} (1978), 101--104.

\bibitem{tondeur:1997} \textsc{P. Tondeur}, \textit{Geometry of foliations}, Monogr. Math. \textbf{90}, Birkh\"auser Verlag, Basel (1997).

\bibitem{wadsley:1975} \textsc{A. W. Wadsley}, Geodesic foliations by circles, \textit{J. Differential Geometry} \textbf{10} (1975), 541--549.


\end{thebibliography}
\end{document}